\documentclass[12pt]{article}
\usepackage{amsmath,amssymb,amstext,dsfont,fancyvrb,float,fontenc,graphicx,caption,subcaption,theorem,hyperref}
\usepackage[utf8]{inputenc}

\usepackage[letterpaper]{geometry}
\ifx\volno\undefined\def\volno{0}\fi
\ifx\volyear\undefined\def\volyear{2017}\fi
\ifx\pagno\undefined\def\pagno{000--000}\fi

\newfont{\footsc}{cmcsc10 at 8truept}
\newfont{\footbf}{cmbx10 at 8truept}
\newfont{\footrm}{cmr10 at 10truept}

\usepackage{fancyhdr}
\usepackage{relsize}
\usepackage{sectsty}
\allsectionsfont{\larger[-1]} 

\renewcommand\paragraph{\@startsection{paragraph}{4}{\z@}
                                    {2ex \@plus.5ex \@minus.2ex}
                                    {-1em}
                                    {\normalfont\normalsize\bfseries}}

\renewcommand\subparagraph{\@startsection{subparagraph}{5}{\parindent}
                                       {2ex \@plus.5ex \@minus .2ex}
                                       {-1em}
                                      {\normalfont\normalsize\bfseries}}

\newlength{\BiblioSpacing}
\renewenvironment{thebibliography}[1]{
\begin{oldthebibliography}{#1}
\setlength{\parskip}{\BiblioSpacing}
\setlength{\itemsep}{\BiblioSpacing}
}
{
\end{oldthebibliography}
}

\usepackage[strict]{changepage}
\def\abstractname{Abstract -}   
\def\abstract{\begin{adjustwidth}{1cm}{1cm} \par    \footnotesize \noindent {\bf \abstractname} 
\def\endabstract{ \end{adjustwidth} \smallskip }}

{\theorembodyfont{\itshape}\newtheorem{theorem}{Theorem}[section]}
{\theorembodyfont{\itshape}}
{\theorembodyfont{\itshape}\newtheorem{definition}[theorem]{Definition}}
{\theorembodyfont{\itshape}\newtheorem{lemma}[theorem]{Lemma}}
{\theorembodyfont{\itshape}}
{\theorembodyfont{\rm}}
{\theorembodyfont{\rm}\newtheorem{remark}[theorem]{Remark}}
{\theorembodyfont{\rm}}
{\theorembodyfont{\rm }}

\newcommand{\N}{\mathbb{N}}
\newcommand{\R}{\mathbb{R}}
\newcommand{\xstar}{x_*}
\newcommand{\rstar}{r_*}
\newcommand{\FD}{D}

\usepackage{multirow}

\title{\Large\bf Superconvergence of Centered Finite Difference Approximations}
\author{\sc 
M. J. Bencomo, 
J. Igot, and E. Maltes}

\begin{document}
\setcounter{page}{1}
\date{February 7, 2026}
\maketitle
\thispagestyle{fancy}

\vskip 1.5em

\begin{abstract}
The \textit{finite difference} (FD) method is commonly used to approximate derivatives of smooth functions, with accuracy typically determined by stencil size and derivative order. 
However, certain centered stencils exhibit unexpectedly higher accuracy, a phenomenon known as \textit{superconvergence}, which has been observed in practice but lacks rigorous explanation.
We present a mathematical framework for superconvergence in centered FD approximations based on Taylor expansions of the truncation error and the resulting linear system for the FD coefficients. 
By analyzing symmetry properties of these coefficients and their interaction with the parity of the derivative order, we identify conditions under which higher-order error terms cancel. 
We show that superconvergence occurs for odd-order derivatives with even centered stencils and for even-order derivatives with odd centered stencils, while no superconvergence occurs for even derivatives with even centered stencils. 
Numerical experiments in MATLAB confirm the predicted convergence rates.
\end{abstract}
 
\begin{keywords}
finite difference approximations; centered stencils; superconvergence; numerical analysis; Vandermonde systems
\end{keywords}

\begin{MSC}
65D25  
\end{MSC}

\section{Introduction} 

\textit{Finite difference} (FD) approximations estimate derivatives of a function by forming linear combinations of the function's values at selected discrete points.
For example, consider a function $f:(a,b)\to \mathbb R$ that is continuously differentiable over $(a,b)$ and let $\xstar\in(a,b)$.
The simplest FD approximation follows from the definition of the derivative, 
\begin{equation}\label{eq:simple_FD}
	f'(\xstar) \approx \frac{f(\xstar+h)-f(\xstar)}{h},
\end{equation}
where we do not take the limit as $h\to 0$ but instead evaluate the expression for some small $h>0$.
The FD approximation in (\ref{eq:simple_FD}) makes use of function values at $\xstar$ and $\xstar+h$, which are referred to as the \textit{stencil} of the FD approximation. 
Using a Taylor series expansion on $f(\xstar+h)$ centered at $\xstar$, it can be shown that
\begin{equation}\label{eq:simple_err}
	\left|f'(\xstar) - \frac{f(\xstar+h)-f(\xstar)}{h} \right| \le Ch
\end{equation}
for $h$ small enough, where $C>0$ is some constant independent of $h$.
Equation~(\ref{eq:simple_err}) implies that as we decrease $h$, the error decreases at least linearly in the limit as $h\to 0$. 
The power of $h$ in the error bound is called the \textit{accuracy order}.

In general, a minimum number of points
\begin{equation}\label{eq:conv_rule}
	N=k+r
\end{equation}
in the stencil is required to achieve an accuracy order of $r$ for the $k^{th}$ derivative.
\textit{Superconvergence} occurs when there are exceptions to this rule, that is, a convergence higher than $N-k$ is observed for an $N$-point FD approximation of the $k^{th}$ derivative. 
Superconvergence may occur for special functions at special points $\xstar$, or if the stencil is chosen in a special manner.
FD approximations with stencils ``centered'' about $\xstar$ are widely known among practitioners to exhibit superconvergence, though not formally proved in the literature as far as the authors are aware, with the exception of one paper by Sadiq and Viswanath (\cite{sadiq2014}).
Their approach is general in that they consider potentially complex stencils and demonstrate under what conditions one could expect an accuracy order of $N-k+b$, for a stencil with $N$ points for the $k^{th}$ derivative where $b$ is a positive integer.
In the case of real stencil points, they demonstrate that the accuracy order cannot be boosted by more than $1$.

\subsubsection*{Main Contribution} 
The primary purpose and contribution of this paper is to present an analysis tailored to a family of well-known centered FD approximations.
This specialized focus enables a proof strategy based entirely on linear algebra, thereby improving accessibility.
Moreover, it allows us to recover and clarify superconvergence phenomena that are obscured in more general frameworks.
Of note is the parity between derivative order $k$ and the number of stencil points $N$, and symmetry and skew-symmetry of coefficients in these approximations.
We prove the existence and uniqueness of the FD coefficients for a given stencil and use this structure to identify when cancellation of error terms leads to superconvergence.
Our main result shows that superconvergence occurs only when approximating odd-order derivatives with even centered stencils, and even-order derivatives with odd centered stencils. 

We begin by introducing notation and establishing standard convergence results for finite difference approximations.
We then describe the class of centered stencils considered in this work, define symmetric and skew-symmetric coefficients, and present our proofs for superconvergence.
Numerical experiments are presented next to illustrate and support the theoretical results.
Finally, we conclude by summarizing our findings.

\section{Theory} \label{sec: theory}
\label{sec:main}

We begin this section by establishing some notation and presenting a proof for standard convergence of FD approximations that will lay the ground work for our other results.

\subsection{Notation and standard convergence}

Let $f:(a,b)\to\R$ be a sufficiently differentiable function over some open interval $(a,b)$, $\xstar\in(a,b)$, and $k, N\in\N$ with $N-k\ge 1$.
Suppose that we wish to approximate $f^{(k)}(\xstar)$ using the following general $N$-point FD approximation,
\[
	f^{(k)}(\xstar) \approx \FD^kf(\xstar):= \sum_{n=1}^N c_n f(x_n).
\]
The set of points $\{x_n\}_{n=1}^N$ is referred to as the \textit{stencil} of the FD approximation and assumed to consist of distinct, ordered points contained in $(a,b)$;
\[
	a< x_1 < x_2 < \cdots  < x_{N-1} < x_{N} < b.
\]
In general, a stencil may not necessarily consist of uniformly spaced points, or include the point $\xstar$ for that matter.
Nonetheless, each stencil point $x_n$ may be expressed relative to $\xstar$ as
\[
	x_n = \xstar+a_n h, \quad \forall n=1,2,...,N,
\]
for some $a_n\in\R$ independent of $h>0$.
In the rest of the paper, we refer to $h$ as the \textit{characteristic size} of the stencil (or simply the \textit{stencil size}) and $\{a_n\}_{n=1}^N$ as the \textit{fundamental stencil} of the FD approximation.

We list some well known FD approximations using our notation.
\begin{itemize}
\item \textit{Forward FD approximation} to the first derivative, where $a_1=0$ and $a_2=1$:
\begin{align*}
	\FD^1f(\xstar) 
    	&= \frac{f(\xstar+h)-f(\xstar)}{h}\\
    	&= -\frac{1}{h}f(\xstar) + \frac{1}{h}f(\xstar+h).
\end{align*} 
\item \textit{Centered FD approximation} to the first derivative, where $a_1=-1$ and $a_2=1$:
\begin{align*}
	\FD^1f(\xstar) 
	&= \frac{f(\xstar+h)-f(\xstar-h)}{2h}\\
	&= -\frac{1}{2h}f(\xstar-h) + \frac{1}{2h}f(\xstar+h).
\end{align*}
\item \textit{Centered FD approximation} to the second derivative, where $a_1=-1$, $a_2=0$, and $a_3=1$:
\begin{align*}
	\FD^2f(\xstar) 
	&=\frac{f(\xstar+h)-2f(\xstar)+f(\xstar-h)}{h^2}\\
	&= \frac{1}{h^2}f(\xstar-h)-\frac{2}{h^2}f(\xstar) + \frac{1}{h^2}f(\xstar+h). 
\end{align*}
\end{itemize} 

\begin{theorem}\label{thm:stand_conv}
\text{\rm (Standard Convergence)}\\
Let $f:(a,b)\to\R$ be a sufficiently differentiable function over an open interval $(a,b)$, and $\xstar\in(a,b)$.
Suppose $k,N\in\N$ with $N-k\ge 1$.
Then, for a given fundamental stencil $\{a_n\}_{n=1}^N$ and stencil size $h>0$, there exists a unique set of coefficients $\{c_n\}_{n=1}^N$ such that
\begin{equation}\label{eq:FDappx}
	\FD^k f(\xstar) := \sum_{n=1}^N c_n f(\xstar+a_nh),
\end{equation}
is an FD approximation of $f^{(k)}(\xstar)$ with accuracy order of at least $N-k$.
In other words,
\[
	\left| \FD^k f(\xstar) - f^{(k)}(\xstar) \right| =\mathcal{O}(h^{N-k}).
\]
\end{theorem}
\begin{proof}
First, we expand each $f(\xstar+a_nh)$ in (\ref{eq:FDappx}) as a Taylor polynomial of degree $N-1$ centered at $\xstar$ with a remainder term in Lagrange form,
\begin{equation}\label{eq:TayExp}
	f(\xstar+a_nh) = 
	\sum_{m=0}^{N-1} f^{(m)}(\xstar)\frac{(a_nh)^m}{m!} + f^{(N)}(\xi_n)\frac{(a_n h)^{N}}{N!},
\end{equation}
for some $\xi_n$ between $\xstar$ and $\xstar+a_nh$.
Let $E(h)$ denote the FD error,
\[
	E(h) := \left| \FD^k f(\xstar) - f^{(k)}(\xstar) \right|.
\]
Plugging in Taylor expansions from (\ref{eq:TayExp}) into the definition of $\FD^k f(\xstar)$ and into $E(h)$ yields
\begin{equation}\label{eq:TayErr}
\begin{split}
	E(h) = \left| f^{(k)}(\xstar) 
	\left( \frac{h^k}{k!} \sum_{n=1}^N c_na_n^k -1 \right) \right.
        + \sum_{m=0,m\neq k}^{N-1} f^{(m)}(\xstar) \frac{h^m}{m!}\left(\sum_{n=1}^N c_n a_n^m\right) 
        \hspace{5em}  & \\
        \left.
        + \frac{h^{N}}{N!}\sum_{n=1}^N f^{(N)}(\xi_n)c_n a_n^{N}
        \right|.
\end{split}
\end{equation}

In order for the FD approximation to be consistent, it follows from (\ref{eq:TayErr}) that it will be necessary for
\begin{equation}\label{eq:Cond1}
	\frac{h^k}{k!} \sum_{n=1}^N c_na_n^k -1 = 0.
\end{equation}
Note that (\ref{eq:Cond1}) implies that $c_n$ will be proportional to $h^{-k}$.
To simplify the system of linear equations to come, we assume 
\begin{equation}\label{eq:tilde c}
	c_n = \tilde c_n \frac{k!}{h^k},
\end{equation}
for all $n=1,2,...,N$, for some $\tilde c_n$ independent of $h$.
The $\tilde c_n$ can be thought of as the stencil-size-independent coefficients of the FD approximation.

In terms of $\tilde c_n$, (\ref{eq:TayErr}) simplifies to
\begin{equation}\label{eq:TayErr2}
\begin{split}
	E(h) = \left|
	f^{(k)}(\xstar)
        \left( \sum_{n=1}^N \tilde c_n a_n^k -1 \right) \right.
        + \sum_{m=0,m\neq k}^{N-1} f^{(m)}(\xstar) \frac{h^{m-k}\; k!}{m!}\left(\sum_{n=1}^N \tilde c_n a_n^m\right) 
        \hspace{3em}  & \\
        \left.
        + \frac{h^{N-k}\; k!}{N!}\sum_{n=1}^N f^{(N)}(\xi_n) \tilde c_n a_n^{N}
    	\right|.
\end{split}
\end{equation}
For a stencil with $N$ points, we can eliminate the first $N$ terms in (\ref{eq:TayErr2}) by having $\tilde c_n$ satisfy the following system of linear equations expressed in matrix-vector form:
\[
	A\mathbf c = \mathbf e_{k+1}
\]
where
\begin{equation}\label{eq:MatA}
	A =
	\begin{bmatrix}
		1 & 1 & \cdots & 1\\
		a_1 & a_2 & \cdots & a_N\\
		a_1^2 & a_2^2 & \cdots & a_N^2\\
		\vdots & \vdots & \ddots & \vdots\\
		a_1^{N-1} & a_2^{N-1} & \cdots & a_N^{N-1}
	\end{bmatrix} \in\R^{N\times N},
	\;
	\mathbf c = \begin{bmatrix}
		\tilde c_1\\ \tilde c_2\\ \vdots \\ \tilde c_N
	\end{bmatrix} \in\R^{N}
\end{equation}
and $\mathbf e_{k+1}\in\R^{N}$ is the standard unit vector with $1$ in the $(k+1)^{th}$ entry and zeros elsewhere.
Note that $A$ matrix in (\ref{eq:MatA}) is a square Vandermonde matrix which is guaranteed to be invertible (though highly ill-conditioned as $N$ increases) given that $a_i\neq a_j$ for all $i,j=1,2,...,N$, for any fundamental stencil.\footnote{See Appendix~\ref{app:Vander} for definitions and results on Vandermonde matrices.}
Invertibility of $A$ thus implies the existence of unique coefficients $\{c_n\}_{n=1}^N$ resulting in 
\begin{equation} \label{eq:errStand}
	E(h) = \left| 
	\frac{h^{N-k}\; k!}{N!} \sum_{n=1}^N f^{(N)}(\xi_n)\tilde c_n a_n^N
	\right|.
\end{equation}
Lastly, one can bound the error as follows,
\[
	E(h) \le Ch^{N-k}
\]
where $C$ is independent of $h$,
\[
	C = \frac{k!}{N!} \sum_{n=1}^N \left| \tilde c_n a_n^N \right| \; \max_{x\in(a,b)}\left| f^{(N)}(x) \right|.
\]
\end{proof}

\begin{remark}
Big-O notation, as used in the statement of Theorem~\ref{thm:stand_conv}, guarantees a lower bound on the order of accuracy, but does not exclude the possibility of higher-order error cancellation under special conditions, such as superconvergence or for special functions.
We further emphasize this distinction in our analysis of superconvergence, since we prove not a lower bound but the exact order of accuracy.
\end{remark}

\subsection{Centered stencils and superconvergence}

Suppose the number of stencil points $N$ is odd, thus $N=2M+1$ for some $M\in\N$.
We re-index the stencil points $x_n$ such that $n$ varies from $-M$ to $M$.
If $N>1$ is even, we re-index in a similar manner but exclude $n=0$.
With this re-indexing, we define what it means for a stencil to be \textit{centered}.

\begin{definition}\label{def:cent_sten}
Let $N\in\N$, with $N=2M+1$ or $N=2M$ for some $M\in\N$. 
An $N$-point stencil $\{x_n\}$ with fundamental stencil $\{a_n\}$ is said to be {centered} if 
\[
	a_{-n} = -a_n, \quad \forall n=1,...,M,
\]
and, in the case $N$ is odd, $a_0=0$.
\end{definition}

If $N$ is odd, we refer to centered stencils as \textit{odd-centered}.
Likewise, if $N$ is even, we use the term \textit{even-centered}.
Note that $a_0=0$ implies $x_0=\xstar$.
The centered FD approximation to the first and second derivatives are examples of even-centered and odd-centered stencils, respectively.

Given our re-indexing for odd-centered and even-centered stencils, we introduce the concepts of \textit{symmetric} and \textit{skew-symmetric} coefficients which will come into play when proving superconvergence of centered FD approximations.

\begin{definition}\label{def:symm_coeff}
Let $N\in\N$, with $N=2M+1$ or $N=2M$ for some $M\in\N$.
The coefficients $\{c_n\}$ of an $N$-point centered FD approximation are said to be {symmetric} if 
\[
	c_{-n} = c_n, \;\forall n=1,2,...,M,
\]
or {skew-symmetric} if
\[
	c_{-n}=-c_n, \; \forall n=1,2,...,M.
\]
\end{definition}

\begin{remark}\label{rm:balanced}
Some authors, like Sadiq and Viswanath \cite{sadiq2014}, refer to stencils in our Definition~\ref{def:cent_sten} as \textit{symmetric}.
More broadly, they use the term \textit{centered} to refer to stencils such that
\begin{equation}\label{eq:balanced}
	\sum _{n=1}^{N} a_n=0 \; \implies \; \frac{1}{N}\sum_{n=1}^N x_n = \xstar.
\end{equation}
As to avoid confusion with our concept of symmetric coefficients, and to be consistent with the established use of  ``centered'' when referring to FD approximations, we stick to our usage of ``centered'' for stencils.
Moreover, we refer to stencils that satisfy (\ref{eq:balanced}) as \textit{balanced}.
\end{remark}

We now prove the superconvergence of centered FD approximations for the cases with opposite parity between derivative order and  number of stencil points.
We show that the order of accuracy is exactly $N-k+1$, and make use of Big-Theta notation.\footnote{See Appendinx~\ref{app:asymp} for more on Big-Theta notation and some useful lemmas.}

\begin{theorem} \label{thm:SCodd}
\text{\rm (Odd-derivative, even-centered stencil case)}\\
Let $f:(a,b)\to\R$ be a sufficiently differentiable function over an open interval $(a,b)$, and $\xstar\in (a,b)$.
Suppose $k=2\ell+1$ and $N=2M$ for some $\ell\in\N_0$ and $M\in\N$, assuming $N-k\ge 1$.
Then, for a given $N$-point, even-centered stencil, with fundamental stencil $\{a_n\}$ and stencil size $h>0$, there exists a unique set of coefficients $\{c_n\}$ such that 
\[
	\FD^k f(\xstar) := \sum_{n=1}^M\Big[ c_{-n}f(\xstar+a_{-n}h)+c_nf(\xstar+a_nh) \Big],
\]
is an FD approximation of $f^{(k)}(\xstar)$ with accuracy order exactly $N-k+1$, assuming $f^{(N+1)}(\xstar)\neq 0$.
In other words, 
\[
	\left| \FD^k f(\xstar) - f^{(k)}(\xstar) \right| = \Theta(h^{N-k+1}).
\]
Furthermore, $\{c_n\}$ must be skew-symmetric.
\end{theorem}
\begin{proof}  
To prove superconvergence of even-centered FD approximations for odd-derivatives we make use of a ``fake'' odd-centered stencil.
Given an $N$-point, even-centered stencil, where $N=2M$, with fundamental stencil $\{a_{-n},a_n\}_{n=1}^M$, consider the resulting $(N+1)$-point odd-centered fundamental stencil given by $\{a_n\}_{n=-M}^M$ where $a_0=0$.
Let $\tilde \FD^k f(\xstar)$ denote the $(N+1)$-point FD approximation
\begin{equation}\label{eq:fake}
	\tilde\FD^k f(\xstar) := \sum_{n=-M}^M c_nf(\xstar+a_nh).
\end{equation}
We first show that there exists a unique set of skew-symmetric coefficients $\{c_n\}_{n=-M}^M$ with $c_0=0$ for which, according to Theorem~\ref{thm:stand_conv}, $\tilde\FD^k f(\xstar)$ is an approximation of order at least $(N+1)-k$.
Since $c_0=0$, this implies that $\tilde \FD^k f\equiv D^kf$.
In other words, the $(N+1)$-point FD approximation in reality has $N$ stencil points, thus implying superconvergence.
Skew-symmetry of the coefficients will come into play when showing the accuracy order does not exceed $N-k+1$.

Assume coefficients $\{c_n\}_{n=-M}^M$ are skew-symmetric, and $c_0=0$ for the ($N+1$)-point, odd-centered FD approximation $\tilde\FD^kf(\xstar)$ in (\ref{eq:fake}).
Similar to Theorem~\ref{thm:stand_conv}, for an $(N+1)$-point stencil, we can eliminate the first $N+1$ terms of the Taylor expansion of the error $E(h)$ by imposing a system of $N+1$ equations on the coefficients $c_n$ of the form
\begin{equation}\label{eq:big}
	A\mathbf c = \mathbf e_{k+1}
\end{equation}
where $\mathbf e_{k+1}\in \mathbb R^{N+1}$ and
\begin{equation}\label{eq:MatAOE}
	A =
	\begin{bmatrix}
		1 &\cdots & 1 &1 & 1 & \cdots & 1\\
		a_{-M}& \cdots& a_{-1}& a_0 & a_1 & \cdots & a_M\\
		a_{-M}^2& \cdots& a_{-1}^2& a_0^2 & a_1^2 & \cdots & a_M^2\\
		\vdots & \ddots & \vdots & \vdots & \vdots & \ddots & \vdots\\
		a_{-M}^{N}&\cdots & a_{-1}^{N}& a_0^N & a_1^{N} & \cdots & a_M^{N}
	\end{bmatrix}\in\mathbb R^{(N+1)\times (N+1)}, \;
	\mathbf c = \begin{bmatrix}
		\tilde c_{-M}\\  \vdots \\ \tilde c_{-1}\\ \tilde c_0\\ \tilde c_1\\ \vdots \\ \tilde c_{M}
	\end{bmatrix}\in\mathbb R^{N+1}.
\end{equation}
Recall, $c_n$ and $\tilde c_n$ are related via (\ref{eq:tilde c}).

Consider the $i^{th}$ component of the product $A\mathbf c$, which we can express as follows given our choice of indexing:
\begin{equation}\label{eq:Ac_i}
    [A\mathbf c]_i 
    = \sum_{n=-M}^M \tilde c_n a_n^{i-1}
    = \tilde c_0 a_0^{i-1} 
    + \sum_{n=1}^M \Big( \tilde c_{-n}a_{-n}^{i-1}+\tilde c_n a_{n}^{i-1}\Big)
\end{equation}
for $i=1,2,...,N+1$.
Then, from the skew-symmetry of the coefficients ($\tilde c_{-n}=-\tilde c_n$), $\tilde c_0 = 0$, and the centeredness of the fundamental stencil ($a_{-n}=-a_n$ and $a_0=0$), (\ref{eq:Ac_i}) implies 
\begin{subequations}
\begin{align}
	[A\mathbf c]_i 
	&= \sum_{n=1}^M \Big( -\tilde c_{n}a_{n}^{i-1}+\tilde c_n a_{n}^{i-1}\Big) = 0, \quad \text{for odd } i, \label{eq:Ac=0}\\
	[A\mathbf c]_i 
	&= 2\sum_{n=1}^M \tilde c_n a_{n}^{i-1}, \quad \text{for even } i. \label{eq:Ac=2}
\end{align}
\end{subequations}

From (\ref{eq:Ac=0}) we can infer that the equations with odd indexes in the original system (\ref{eq:big}) are automatically satisfied.
Using (\ref{eq:Ac=2}) we arrive at a smaller system of equations to determine $\{\tilde c_n\}_{n=1}^M$:
\begin{equation}\label{eq:small}
    \hat A \hat{\mathbf{c}}=\frac{1}{2}\mathbf e_{\ell+1}
\end{equation}
where $\mathbf e_{\ell+1}\in\mathbb R^{M}$ (recall $k=2\ell+1$) and
\[
    \hat A = 
    \begin{bmatrix} 
        a_1 & a_{2} & \cdots & a_{M} \\
        a_1^3 &a_2^3 & \cdots & a_M^3\\
        \vdots & \vdots &\ddots &\vdots & \\
        a_1^{2M-1} &a_2^{2M-1} & \cdots & a_M^{2M-1}\\  
    \end{bmatrix}\in\mathbb R^{M\times M}, \;
    \hat{\mathbf c} =
    \begin{bmatrix}
        \tilde c_1\\
        \tilde c_2\\
        \vdots\\
        \tilde c_M \\
    \end{bmatrix}\in\mathbb R^M.
\]
Note that $\hat A$ is a generalized Vandermonde matrix of the form
\[
	[\hat A]_{ij} = a_j^{z_i},
\]
with $0<a_1<a_2<\cdots < a_M$, where $z_i = 2i-1$ for $i=1,2,...,M$ and thus
\[
	0 \le z_1 < z_2 < \cdots < z_M.
\]
It follows from Lemma~\ref{lem:genVander} that $\hat A$ is invertible.
From the unique solution to reduced system (\ref{eq:small}), we can construct a unique set of skew-symmetric coefficients $\{c_n\}_{n=-M}^M$ so that the original system (\ref{eq:big}) is satisfied.

To show that the accuracy order is exactly $N+1-k$, we look at the expansion of the error $E(h)$ up to $N+2$ terms, after having eliminated the first $N+1$ terms:
\begin{equation*}
    E(h) = \left| K_1h^{N+1-k} + K_2(h)h^{N+2-k} \right|
\end{equation*}
where
\[
	K_1 = \frac{k!}{(N+1)!}f^{(N+1)}(\xstar) \sum_{n=-M}^M \tilde c_n a_n^{N+1}
\]
is independent of $h$, and
\[
	K_2(h) = \frac{k!}{(N+2)!} \sum_{n=-M}^{M} f^{(N+2)}(\xi_n)\tilde c_na_n^{N+2}.
\]
The dependency of $K_2$ on $h$ is implicit via the $\xi_n$.
Using Lemma~\ref{lem:BigTheta2}, it suffices to show that $K_1\neq 0$ and $K_2(h)$ is bounded for small enough $h>0$.
In fact, if $h$ is small enough so that the resulting stencil in contained in $(a,b)$ then we can bound $K_2(h)$ as follows:
\[
	|K_2(h)| \le \frac{k!}{(N+2)!}\sum_{n=-M}^M |\tilde c_na_n^{N+2}| \; \max_{x\in(a,b)}\left|f^{(N+2)}(x)\right|.
\]

Assuming $f^{(N+1)}(\xstar)\neq 0$, using skew-symmetry of coefficients, centeredness of the stencil, and the fact that $N$ is even, we have that $K_1\neq 0$ if and only if
\begin{equation}\label{eq:cr}
    \sum_{n=-M}^M \tilde c_n a_n^{N+1} 
    = 2\sum_{n=1}^M \tilde c_n a_n^{N+1}
    = 2 \hat{\mathbf c}^T \mathbf \rstar
    \neq 0,
\end{equation}
where
\[
	\mathbf \rstar =
	\begin{bmatrix}
	a_1^{N+1} &a_2^{N+1} & \cdots &a_M^{N+1}
	\end{bmatrix}^T.
\] 
Recall that the matrix $\hat A\in\R^{M\times M}$ was shown to be invertible, and hence $\{\mathbf r_n\}_{n=1}^M$ forms a basis for $\R^M$, where $\mathbf r_n$ is the $n^{th}$ row vector of $\hat A$.
It follows that there exists a unique $\mathbf d\in\R^M$ such that
\[
	\mathbf \rstar = \hat A^T\mathbf d
\]
Moreover, since $\hat A\hat{\mathbf c}=\mathbf e_{\ell+1}$, we have
\[
	\hat{\mathbf c}^T\mathbf \rstar = \hat{\mathbf c}^T\hat A^T\mathbf d = \mathbf e_{\ell+1}^T\mathbf d= d_{\ell+1}.
\]
We show that $K_1\neq 0$ by showing that $d_{\ell+1}\neq 0$.

Suppose, as to get a contradiction, that $d_{\ell+1}=0$, which would imply that $\mathbf \rstar$ is in the span of $\{\mathbf r_n\}_{n\neq\ell+1}$.
Construct the matrix $B\in\R^{M\times M}$ by removing the $(\ell+1)^{th}$ row of $\hat A$ and concatenating the row $\mathbf \rstar^T$.
In other words, 
\[
	B = \begin{bmatrix}
		\hat A^{[\ell+1]}\\ \mathbf \rstar^T
	\end{bmatrix}
\]
where $\hat A^{[\ell+1]}\in\R^{(M-1)\times M}$ is the matrix $\hat A$ with its $(\ell+1)^{th}$ row removed.
Note that $B$ is not invertible since its last row is linearly dependent on the previous rows, assuming $d_{\ell+1}=0$.
However, $B$ is also a generalized Vandermonde matrix of the form
\[
	[B]_{ij} = a_j^{z_i}
\]
with $0<a_1<a_2<\cdots < a_M$, where $z_M = N+1$. 
Since
\[
	0<2i-1< N+1, \quad \forall i=1,2,...,M,
\]
this implies $0 \le z_1 < z_2 < \cdots < z_M$.
It follows from Lemma~\ref{lem:genVander} that $B$ must be invertible, and hence our contradiction.
\end{proof}

\begin{theorem} \label{thm:SCeven}
\text{\rm (Even-derivative, odd-centered stencil case)}\\
Let $f:(a,b)\to\R$ be a sufficiently differentiable function over an open interval $(a,b)$, and $\xstar\in(a,b)$. 
Suppose $k=2\ell$ and $N=2M+1$ for some $\ell, M \in \N$, assuming $N-k\ge 1$. 
Then, for a given $N$-point stencil, with fundamental stencil $\{a_n\}$ and stencil size $h>0$, there exists a unique set of coefficients $\{c_n\}$ such that 
\begin{equation}\label{eq:FD_Nodd}
	\FD^k f(\xstar) := \sum_{n=-M}^Mc_nf(\xstar+a_nh ),
\end{equation}
is an FD approximation of $f^{(k)}(\xstar)$ with accuracy order exactly $N-k+1$, assuming $f^{(N+1)}(\xstar)\neq 0$.
In other words, 
\[
	\left| \FD^k f(\xstar) - f^{(k)}(\xstar) \right| = \Theta(h^{N-k+1}).
\]
Furthermore, $\{c_n\}$ must be symmetric.
\end{theorem}
\begin{proof} 
Unlike Theorem~\ref{thm:SCodd}, we do not make use of a ``fake'' stencil to prove superconvergence.
Instead, we show that $N$-point FD approximation must have symmetric coefficients, and that superconvergence is achieved due to that symmetry.

Assume coefficients $\{c_n\}_{n=-M}^M$ are symmetric for the $N$-point, odd-centered FD approximation in (\ref{eq:FD_Nodd}).
For an $N$-point approximation, we can eliminate the first $N$ terms in the Taylor expansion of the error $E(h)$ by imposing a system of $N$ equations for the coefficients $c_n$ of the form
\begin{equation}\label{eq:Ac=e,even}
	A\mathbf c = \mathbf e_{k+1},
\end{equation}
just as in Theorem~\ref{thm:SCodd}, but with $\mathbf e_{k+1}\in\R^N$ and 
\begin{equation}\label{eq:MatAEO}
	A =
	\begin{bmatrix}
		1 &\cdots &1 & 1  & 1 & \cdots & 1\\
		a_{-M}& \cdots& a_{-1}&a_0 & a_1  & \cdots & a_M\\
		a_{-M}^2& \cdots&  a_{-1}^2&a_0^2 & a_1^2  & \cdots & a_M^2\\
		\vdots & \ddots & \vdots & \vdots & \vdots& \ddots & \vdots\\
		a_{-M}^{N-1}&\cdots & a_{-1}^{N-1} &a_0^{N-1} & a_1^{N-1} & \cdots & a_M^{N-1}
	\end{bmatrix} \in\R^{N\times N},
	\;
	\mathbf c = \begin{bmatrix}
		\tilde c_{-M}\\ \vdots \\ \tilde c_{-1}\\ \tilde c_0\\ \tilde c_1\\ \vdots \\ \tilde c_{M}\\
	\end{bmatrix}\in\R^N.
\end{equation}
Recall, $c_n$ and $\tilde c_n$ are related via (\ref{eq:tilde c}).

Consider the $i^{th}$ component of the product $A\mathbf c$, which is given by (\ref{eq:Ac_i}), for $i=1,2,...,N$.
By symmetry of the coefficients ($\tilde c_{-n}=\tilde c_n$), and centeredness of the fundamental stencil ($a_{-n}=-a_n$ and $a_0=0$), we have
\begin{subequations}
\begin{align}
	[A\mathbf c]_i &= \tilde c_0a_0^{i-1} + \sum_{n=1}^M \Big( -\tilde c_{n}a_{n}^{i-1}+\tilde c_n a_{n}^{i-1}\Big) = 0, \quad \text{for even }\; i, \label{eq:Ac=0 even}\\
	[A\mathbf c]_i &= 2\sum_{n=0}^M \tilde c_n a_{n}^{i-1}, \quad \text{for odd } \; i. \label{eq:Ac=2 even}
\end{align}
\end{subequations}

From (\ref{eq:Ac=0 even}) we can infer that the equations with even indexes in in the original system (\ref{eq:Ac=e,even}) are automatically satisfied.
Using (\ref{eq:Ac=2 even}) we construct a smaller system of equations to determine $\{\tilde c_n\}_{n=0}^M$:
\begin{equation}\label{eq:hatAeven}
	\hat A \hat{\mathbf c} = \frac{1}{2}\mathbf e_{\ell+1}
\end{equation}
where $\mathbf e_{\ell+1}\in\R^{M+1}$ (recall $k=2\ell$) and
\begin{equation} \label{mtrx:Atdl}
	\hat A = 
	\begin{bmatrix} 
		1 &1 &\dots &1\\
		a_1^2 & a_{2}^2 & \cdots & a_{M}^2 \\
		\vdots & \vdots &\ddots &\vdots & \\
		a_1^{2M} & a_{2}^{2M} & \cdots & a_{M}^{2M}\\  
	\end{bmatrix}\in\R^{(M+1)\times (M+1)},\;
	\tilde{\mathbf c} =
	\begin{bmatrix}
		\tilde c_0\\ \tilde c_1\\ \vdots\\ \tilde c_M
	\end{bmatrix}\in\R^{M+1}.
\end{equation}
Just as is Theorem~\ref{thm:SCodd}, it follows that $\hat A$ is invertible.

From the unique solution of the reduced system (\ref{eq:hatAeven}), we can construct a unique set of symmetric coefficients $\{c_n\}_{n=-M}^M$ that satisfy the original system (\ref{eq:Ac=e,even}).
Consequently, by Theorem~\ref{thm:stand_conv}, the $N$-point FD approximation in (\ref{eq:FD_Nodd}) achieves an accuracy order of at least $N-k$.
Achieving an accuracy order of at least $N+k-1$ requires the elimination of the $(N+1)$st term in the Taylor expansion of the error $E(h)$, which imposes the condition
\begin{equation}\label{eq:extra}
	\sum_{n=-M}^M \tilde c_n a_n^{N} = \tilde c_0 a_0^N + \sum_{n=1}^M \Big( \tilde c_{-n}a_{-n}^N + \tilde c_n a_n^N\Big) =0.
\end{equation}
Since $N$ is odd, the fundamental stencil is centered, and the coefficients are symmetric, condition (\ref{eq:extra}) is automatically satisfied.

The remainder of the argument, establishing that the order of accuracy is exactly $N+1-k$ proceeds analogously to the proof of Theorem~\ref{thm:SCodd}, with the modifications that $N$ is odd and the coefficients $\{c_n\}$ are symmetric.
\end{proof}

\subsection{Non-superconvergence for matching parity cases}

In this section we prove that superconvergence is not possible if the parity between derivative order and number of stencil points is the same.
Symmetry, or skew-symmetry, of FD coefficients is still expected and determined by the parity of the derivative order.

\begin{theorem} 
Let $f:(a,b)\to\R$ be a sufficiently differentiable function over an open interval $(a,b)$, and $\xstar\in (a,b)$.
Suppose $k$ and $N$ are positive integers with matching parity, assuming $N-k\ge 1$.
Then, for a given $N$-point, centered stencil, with fundamental stencil $\{a_n\}$ and stencil size $h>0$, there exists a unique set of coefficients $\{c_n\}$ such that 
\[
	\FD^k f(\xstar) := \sum_{n=-M}^M c_nf(\xstar+a_nh), \; \text{(odd $N$)}
\]
or 
\[
	\FD^k f(\xstar) := \sum_{n=1}^M\Big[ c_{-n}f(\xstar+a_{-n}h)+c_nf(\xstar+a_nh) \Big],
 \; \text{(even $N$)}
\]
is an FD approximation of $f^{(k)}(\xstar)$ with accuracy order exactly $N-k$, assuming $f^{(N)}(\xstar)\neq 0$.
In other words, 
\[
	\left| \FD^k f(\xstar) - f^{(k)}(\xstar) \right| = \Theta(h^{N-k}).
\]
Furthermore, $\{c_n\}$ must be symmetric if $k$ is even or skew-symmetric if $k$ is odd.
\end{theorem}
\begin{proof} 
Suppose both $k$ and $N$ are odd, and $N=2M+1$ for some $M\in\N$.
Let $\tilde N = N-1$, which makes $\tilde N$ an even integer.
Just as in the proof for Theorem~\ref{thm:SCodd}, we can construct a ``fake'' $(\tilde N+1)$-point, centered stencil for which there exists a unique set of skew-symmetric coefficients $\{c_n\}_{n=-M}^M$ with $c_0=0$ such that the FD approximation has accuracy order of exactly $\tilde N-k+1$.
The ``fake'' stencil is the actual $N$-point stencil in question for which we have shown has an accuracy order of exactly $(N-1)+k+1=N-k$.

Now suppose that both $k$ and $N$ are even, so that $k=2\ell$ and $N=2M$ for some $\ell,M\in\N$.
We first show that the coefficients $\{c_n\}$ must be symmetric, using arguments similar to those in Theorem~\ref{thm:SCodd}.
Assume $\{c_n\}$ are symmetric for the $N$-point, even-centered FD approximation.
According to Theorem~\ref{thm:stand_conv}, this FD approximation has an accuracy order of at least $N-k$ if $\{c_n\}$ satisfy the following system of equations
\[
	A\mathbf c = \mathbf e_{k+1}
\]
where $\mathbf e_{k+1}\in\R^{N}$ and
\[
	A = \begin{bmatrix}
		1 & \cdots & 1 & 1 & \cdots & 1\\
		a_{-M} & \cdots & a_{-1} & a_{1} & \cdots & a_M\\
		a_{-M}^2 & \cdots & a_{-1}^2 & a_{1}^2 & \cdots & a_M^2\\
		\vdots & \ddots & \vdots & \vdots & \ddots & \vdots\\
		a_{-M}^{N-1} & \cdots & a_{-1}^{N-1} & a_{1}^{N-1} & \cdots & a_M^{N-1}
	\end{bmatrix}\in\R^{N\times N}, \;
	\mathbf c = \begin{bmatrix}
		\tilde c_{-M}\\ \vdots \\ \tilde c_{-1} \\ \tilde c_{1}\\ \vdots \\ \tilde c_{M}
	\end{bmatrix}\in\R^N.
\]
Note the absence of the zeroth index in these equations given that we are working with an even stencil.
Again, $c_n$ and $\tilde c_n$ are related via (\ref{eq:tilde c}).

Using symmetry of the coefficients ($c_{-n}=c_{n}$) and centered nature of the fundamental stencil ($a_{-n}=-a_n$ and $a_0=0$), the $i^{th}$ component of the product $A\mathbf c$ is given by (\ref{eq:Ac=0 even}) and (\ref{eq:Ac=2 even}) in Theorem~\ref{thm:SCeven}, for $i=1,2,...,N$.
In particular, symmetry of $\{c_n\}$ ensures that the even-indexed equations are satisfied.
The remainder of the proof proceeds as in Theorem~\ref{thm:SCeven}, where the coefficients, $\{c_n\}_{n=1}^{M}$ are uniquely determined by constructing a reduced Vandermonde-like system.

To show that the accuracy order is exactly $N-k$, we look at the expansion of the error $E(h)$ up to $N+1$ terms, after having eliminated the first $N$ terms:
\[
	E(h) = \left| K_1h^{N-k} + K_2(h)h^{N+1-k} \right|
\]
where
\[
	K_1 = \frac{k!}{N!}f^{(N)}(\xstar) \sum_{n=1}^M \Big( \tilde c_{-n} a_{-n}^N + \tilde c_n a_n^{N} \Big)
\]
is independent of $h$, and
\[
	K_2(h) = \frac{k!}{(N+1)!} \sum_{n=1}^{M} 
	\Big( 
		f^{(N+1)}(\xi_{-n})\tilde c_{-n}a_{-n}^{N+1} + 
		f^{(N+1)}(\xi_n)\tilde c_na_n^{N+1}
	\Big).
\]
Using Lemma~\ref{lem:BigTheta2}, $K_2(h)$ is bounded and $K_1\ne 0$ in the same manner as in Theorem~\ref{thm:SCodd}, where
\[
	|K_2(h)| \le \frac{k!}{(N+1)!}
	\sum_{n=1}^M 
		2|\tilde c_na_n^{N+1}| 
	\; \max_{x\in(a,b)}\left|f^{(N+1)}(x)\right|.
\]
Assuming $f^{(N)}(\xstar)\neq 0$, we have that $K_1\neq 0$ if and only if
\begin{align*}
	\sum_{n=1}^M \Big( \tilde c_{-n} a_{-n}^{N}  + \tilde c_n a_n^{N} \Big)
	= 2\sum_{n=1}^M \tilde c_n a_n^{N}
	= 2 \hat{\mathbf c}^T \mathbf \rstar
	\neq 0,
\end{align*}
where we have made use of $N$ being even and define
\[
	\hat{\mathbf c} = 
		\begin{bmatrix}
		\tilde c_1 & \tilde c_2&  \cdots & \tilde c_n
		\end{bmatrix}^T, \quad
	\mathbf \rstar = 
		\begin{bmatrix}
		a_1^N & a_2^N & \cdots & a_M^{N}
		\end{bmatrix}^T.
\]
The remainder of the proof follows Theorem~\ref{thm:SCodd} and \ref{thm:SCeven} by showing that $\mathbf \rstar$ has a component in the direction of $\tilde{\mathbf c}$ via contradiction.
\end{proof}

\section{Numerical Experiments}
\label{sec:experiments}

To validate our theoretical results we performed numerical convergence tests on several FD approximations in MATLAB. 
These tests consisted of plotting the FD error $E(h)$ as a function of successively halved values of the stencil size $h$. 
Per stencil, we compute the coefficients $\{c_n\}$ by setting up and solving the problem $A\mathbf c=\mathbf e_{k+1}$ as previously described in the proof of Theorem~\ref{thm:stand_conv}. 
Once the coefficients are obtained, they define the weights for each stencil point, which can then be used to approximate the derivative at various $h$ values. 

\subsection{Experimental Setup}
We used the test function $f(x)=$ cos($\pi x$) with a fixed evaluation point at  $\xstar = 0.3$, for which $f^{(k)}(\xstar)\neq0$ for any $k$.
The initial stencil size was $h_0 = 1/2$ and subsequently halved iteratively;
\[
	h_i = \frac{1}{2^{i+1}}, 
\]
for refinement index $i=0,1,...,8$.
We show results for $N$-point, centered FD approximations for the $k^{th}$ derivative, for several $k$ and $N$ values.
Specifically, we chose the following set of schemes:

\subsubsection*{$1^{st}$-Order Forward FD Approximations:} 
This is our baseline method, the family of \textit{forward} FD approximations of the form
\[
    \FD^k f(\xstar) = \sum_{n=0}^{N-1} c_n f(\xstar+a_nh),
\]
where 
\[
	0 = a_0 < a_1 < \cdots < a_{N-1}
\]
with $N=k+1$ for a given $k$.
Though we do not give a formal proof, it is well known that these $N$-point FD approximations are of 1st order and of standard convergence, thus requiring $N=k+1$. 

\subsubsection*{$r^{th}$-Order Centered FD Approximations:} 
These are the higher-order centered methods which appear in either of the two forms:
\begin{itemize}
    \item Even-centered approximations, with $N=2M$, for some $M\in\mathbb N$,
    \[
        \FD^k f(\xstar) = \sum_{n=1}^{M} \Big[c_{-n} f(\xstar+a_{-n}h)+c_n f(\xstar+a_nh)\Big].
    \] 

    \item Odd-centered approximations, with $N=2M+1$, for some $M\in\mathbb N$,
    \[
        \FD^k f(\xstar) = c_0f(\xstar+a_0h) + \sum_{n=1}^{M}\Big[ c_{-n} f(\xstar+a_{-n}h)+c_n f(\xstar+a_nh)\Big].
    \]
\end{itemize}

In Table~\ref{table:stencil-sketches} we include the fundamental stencils $\{a_n\}$ for all of the FD approximations used in our experiments; each row corresponds to a different approximation.
In our tables and plots, ``FN'' and ``CN'' will refer to the $N$-point forward and centered FD approximations, respectively.
Empty entries are to be interpreted as the fundamental stencil not defined for those indexes.
For simplicity, we take $a_n = n$.

\begin{table}[h!]
\centering
\begin{tabular}{ l@{\hskip 2pt} ||c@{\hskip 2pt}| c@{\hskip 2pt} | c@{\hskip 2pt} | c@{\hskip 2pt} | c@{\hskip 2pt} | c@{\hskip 2pt} | c@{\hskip 2pt} | c@{\hskip 2pt} | c@{\hskip 2pt} }
& $a_{-4}$ & $a_{-3}$ & $a_{-2}$ &$a_{-1}$ & $a_0$& $a_1$ & $a_2$ & $a_3$ & $a_4$ 
\\ \hline \hline
\textbf{F2}&& & & & 0 & 1 &  &  & \\ \hline
\textbf{F3}&& & & & 0 & 1 & 2 &  & \\ \hline
\textbf{F4}&& & & & 0 & 1 & 2 &  3 & \\ \hline
\textbf{F5}&& & & & 0 & 1 & 2 & 3 & 4\\ \hline
\textbf{C2}&& & & $-1$ &  & 1 &  & & \\ \hline
\textbf{C3}&& & & $-1$ & 0 & 1 &  & & \\ \hline
\textbf{C4}&& & $-2$ & $-1$ &  & 1 & 2 & & \\ \hline
\textbf{C5}&& & $-2$ & $-1$ & 0 & 1 & 2 & & \\ \hline
\textbf{C6}&& $-3$ & $-2$ & $-1$ & & 1 & 2 & 3 & \\ \hline
\textbf{C7}&& $-3$ & $-2$ & $-1$ & 0 & 1 & 2 & 3 & \\ \hline
\textbf{C8}& $-4$ & $-3$ & $-2$ & $-1$ &  & 1 & 2 & 3 & 4\\ \hline
\textbf{C9}& $-4$ & $-3$ & $-2$ & $-1$ & 0 & 1 & 2 & 3 & 4\\
\end{tabular}
\caption{Fundamental stencils for various forward and centered approximations.}
\label{table:stencil-sketches}
\end{table}

\subsection{Results}

The convergence plots in Figure~\ref{fig:convergence-plots} present errors $E(h)$ of various FD approximations for derivatives of order $k=1$ to $4$. 
They are plotted in log-log scale to clearly visualize the convergence behavior of each stencil. 
As the error decreases proportionally to a power of the stencil size, log-log transforms this power into a straight line whose slope corresponds to the observed accuracy order.
Overall, the results validate the theoretical orders of the regular convergent and superconvergent FD approximations. 

The forward FD approximations are shown to be $1^{st}$ order accurate, as to be expected given that for each $k$, $N=k+1$ resulting in $r=N-k=1$. 
For centered approximations of $r^{th}$ accuracy order, superconvergence was observed whenever $N$ and $k$ had opposite parity, yielding $r = N-k+1$.
Also, computed coefficients ${c_n}$ were symmetric or skew-symmetric in accordance to the parity of the derivative order.
In the case $N$ and $k$ were both odd, the resulting $N$-point FD approximation is equivalent to an even-centered $(N-1)$-point FD approximation since $c_0=0$.
This explains why certain error curves overlap in Figure~\ref{fig:convergence-plots} for odd derivatives.
For example, C2 and C3 overlap and likewise C4 and C5 for $k=1$.
For $N$ and $k$ even, the centered method exhibited regular convergence with accuracy order $r = N-k$.
Interestingly enough, these approximations were observed to yield higher errors when compared to superconvergent odd-centered approximations of equivalent order.

The specific approximate values of the convergence rates, that is the accuracy order, are tabulated in Table~\ref{fig:convergence_tables},  computed using the formula
\begin{equation}
	\text{Rate at refinement index }i = \frac{\text{ln}(E(h_{i+1}))-\text{ln}(E(h_{i}))}{\text{ln}(h_{i+1})-\text{ln}(h_{i})}.
\end{equation}
Refinement indices 4 to 6 were selected for analysis, as this range exhibits the most linear behavior in log-log scale.
Some deviations from linearity in the convergence plots appear at extreme values of $h$. 
The initial distortions (for larger $h$) are due to pre-asymptotic behavior, while the final distortions (for smaller $h$) can be attributed to well known loss-of-significance errors.

\begin{figure}[!ht]
  \centering
  \begin{minipage}{0.48\textwidth}
    \centering
    \includegraphics[width=\linewidth]{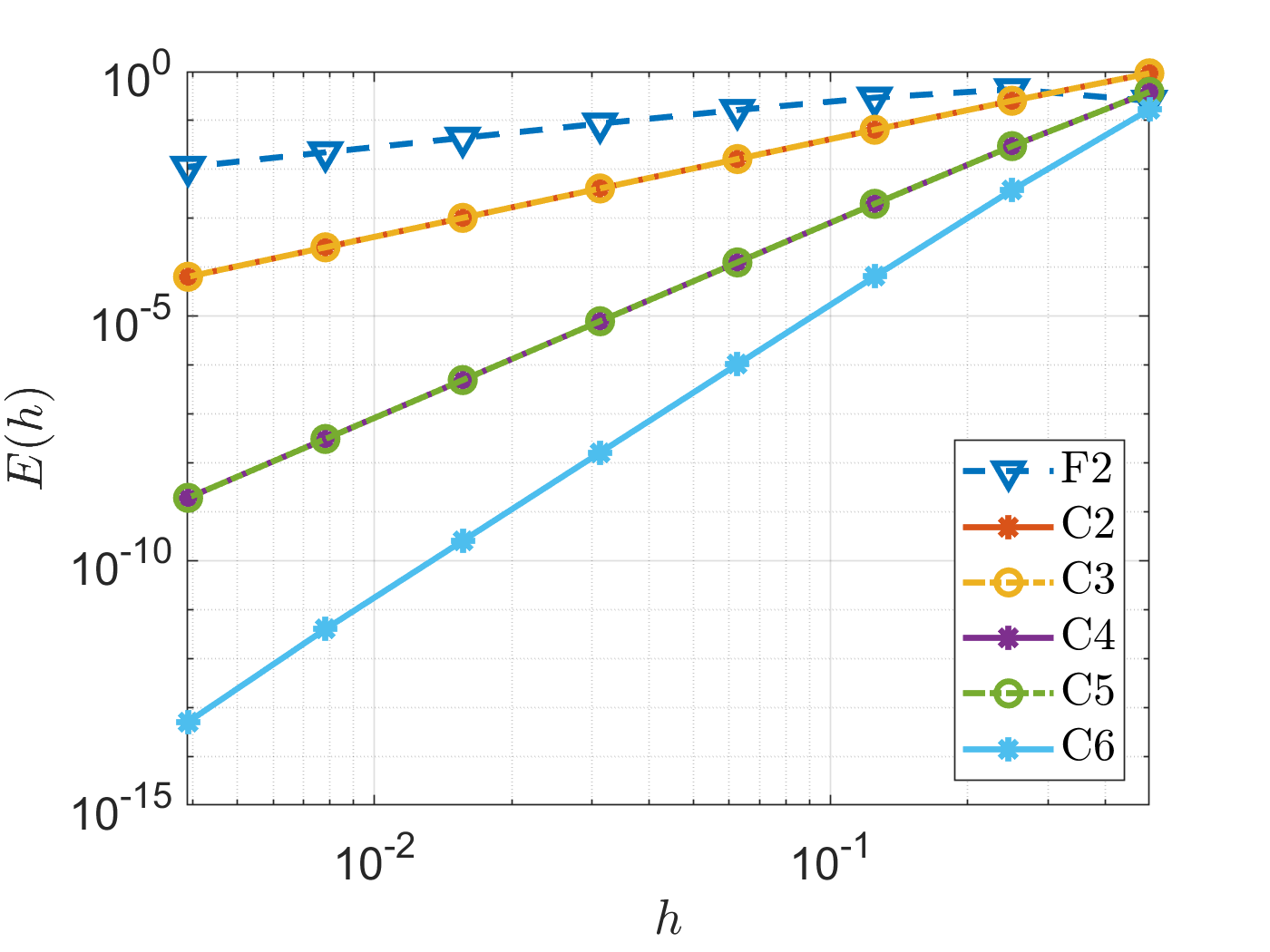}
    \subcaption{$k=1$}
  \end{minipage}
  \hfill
  \begin{minipage}{0.48\textwidth}
    \centering
    \includegraphics[width=\linewidth]{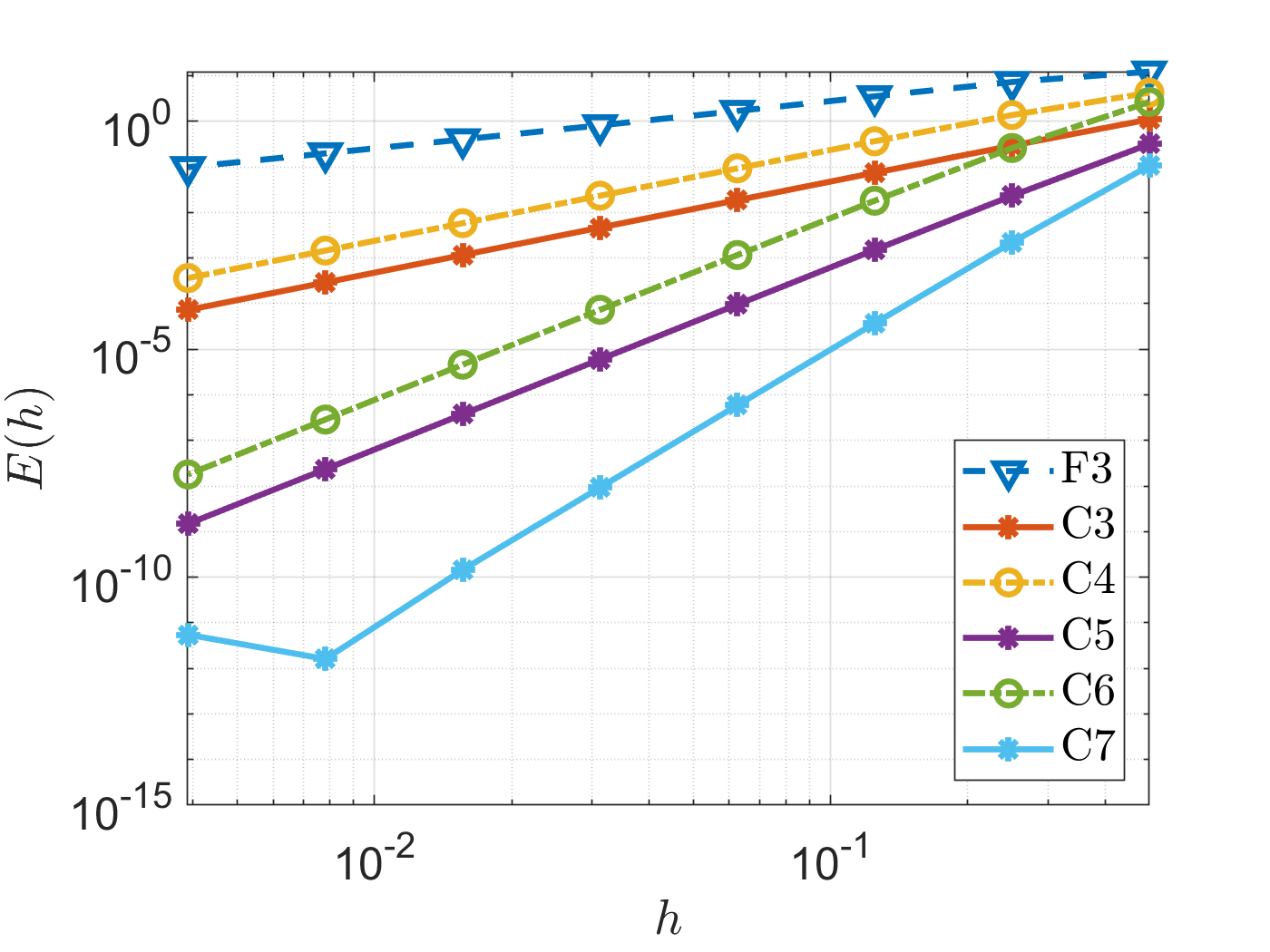}
    \subcaption{$k=2$}
  \end{minipage}

  \vspace{0.5em}

  \begin{minipage}{0.48\textwidth}
    \centering
    \includegraphics[width=\linewidth]{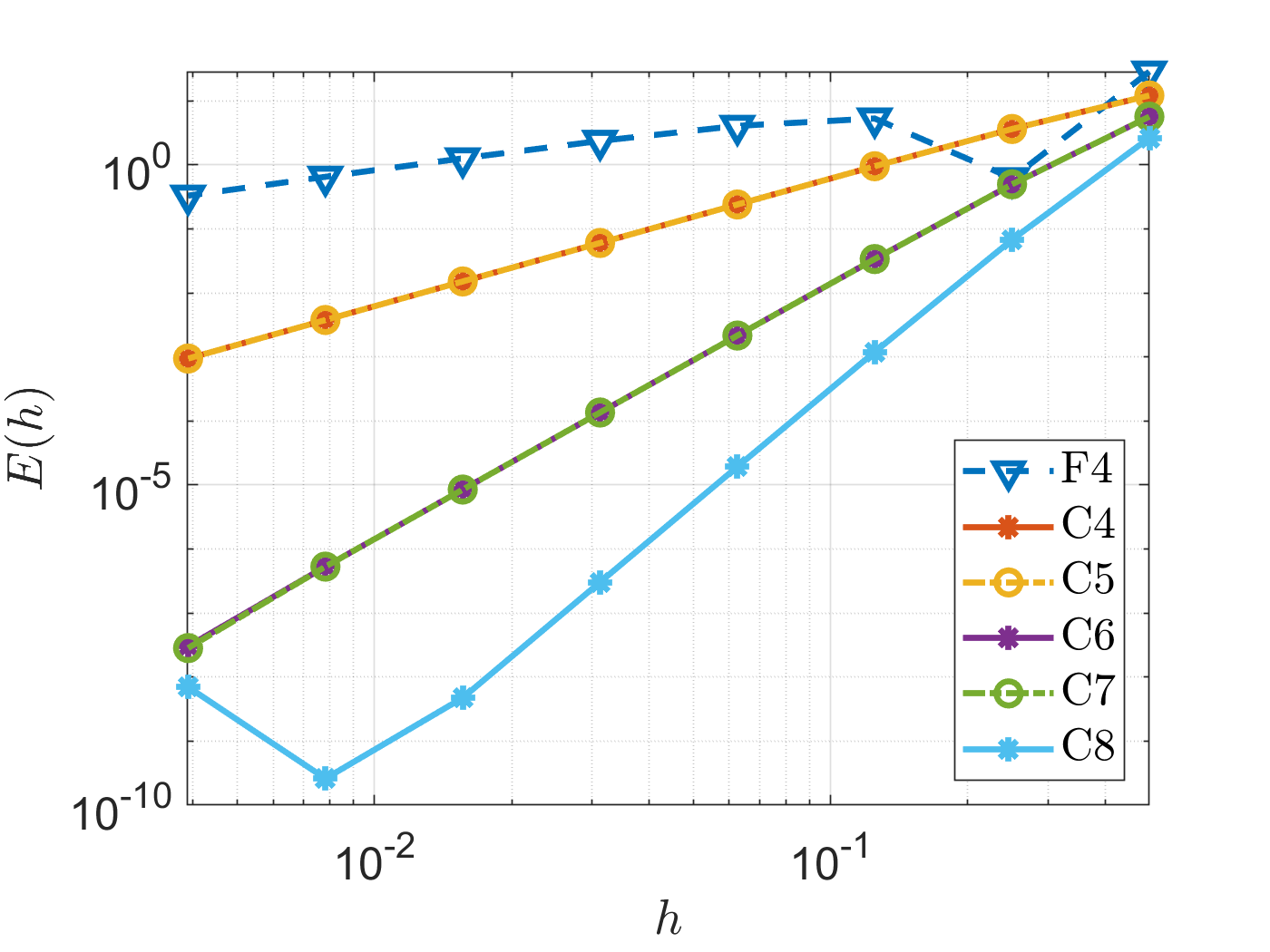}
    \subcaption{$k=3$}
  \end{minipage}
  \hfill
  \begin{minipage}{0.48\textwidth}
    \centering
    \includegraphics[width=\linewidth]{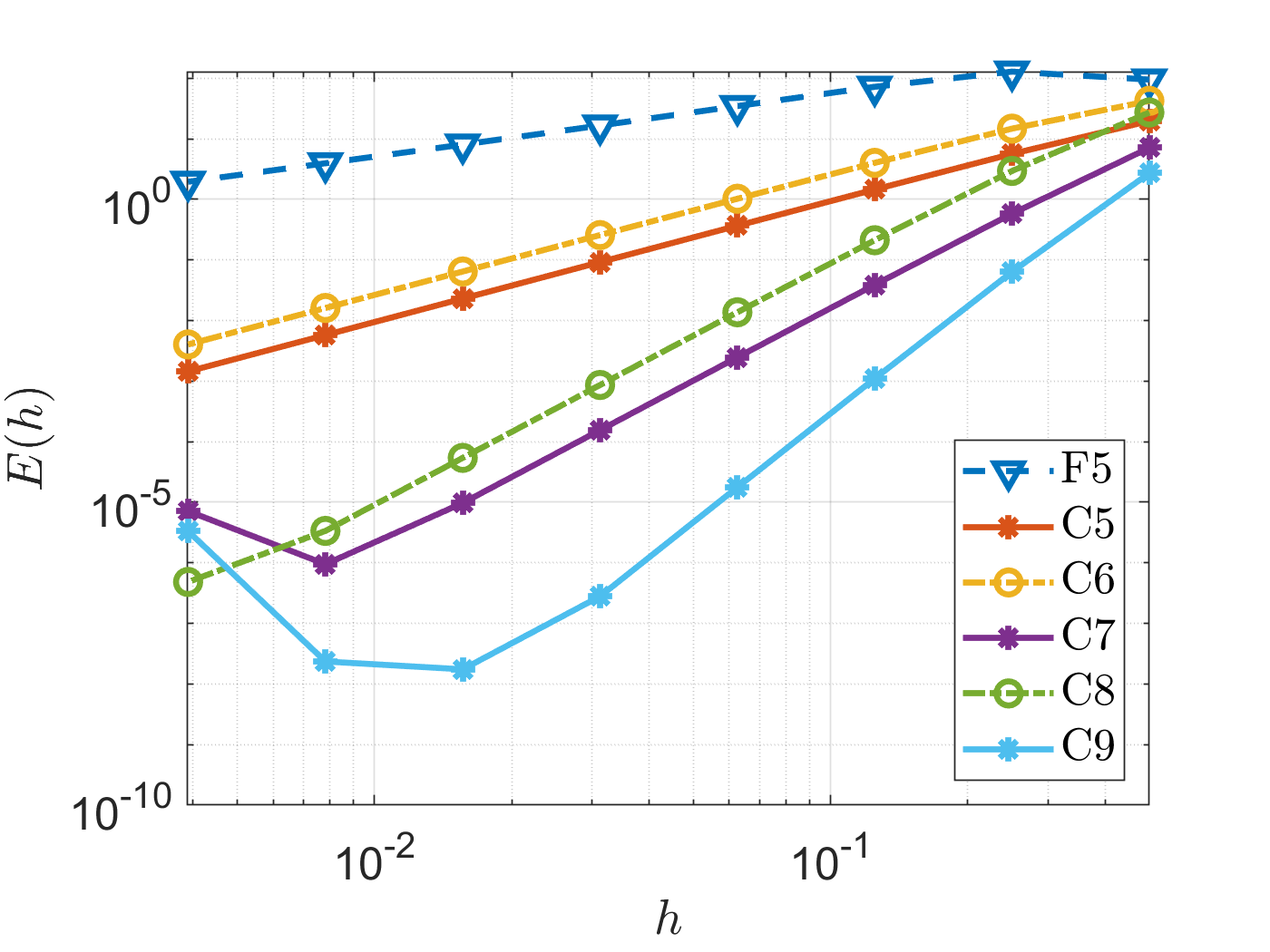}
    \subcaption{$k=4$}
  \end{minipage}
  \caption{FD error $E(h)$ versus stencil size $h$ of $k^{th}$ derivative.}
  \label{fig:convergence-plots}
\end{figure}

\begin{table}[!ht] 
\centering
\renewcommand{\arraystretch}{1.2}

\begin{subtable}[t]{0.45\textwidth}
    \centering
    \begin{tabular}{c|c|c|c|c|}
    \multicolumn{2}{r}{} & \multicolumn{3}{c}{\small\textbf{Refinement Index}} \\
    \cline{3-5}
    \multicolumn{2}{r|}{} & \textbf{4} & \textbf{5} & \textbf{6} \\
    \cline{2-5}
    \multirow{4}{*}{\rotatebox[origin=c]{90}{\small\textbf{Conv. Rate}}}
    & \textbf{F2} & 0.927 & 0.965 & 0.983 \\ \cline{2-5}
    & \textbf{C2} & 1.998 & 1.999 & 2.000 \\ \cline{2-5}
    & \textbf{C4} & 3.995 & 3.999 & 4.000 \\ \cline{2-5}
    & \textbf{C6} & 5.992 & 5.998 & 5.997 \\
    \cline{2-5}
    \end{tabular}
    \caption{k=1}
\end{subtable}
\hspace{1em}
\begin{subtable}[t]{0.45\textwidth}
    \centering
    \begin{tabular}{c|c|c|c|c|}
    \multicolumn{2}{r}{} & \multicolumn{3}{c}{\small\textbf{Refinement Index}} \\
    \cline{3-5}
    \multicolumn{2}{r|}{} & \textbf{4} & \textbf{5} & \textbf{6} \\
    \cline{2-5}
    \multirow{4}{*}{\rotatebox[origin=c]{90}{\small\textbf{Conv. Rate}}}
    & \textbf{F3} & 1.046 & 1.027 & 1.014 \\ \cline{2-5}
    & \textbf{C3} & 1.999 & 2.000 & 2.000 \\ \cline{2-5}
    & \textbf{C4} & 1.994 & 1.999 & 2.000 \\ \cline{2-5}
    & \textbf{C5} & 3.996 & 3.999 & 4.000 \\ \cline{2-5}
    & \textbf{C6} & 3.990 & 3.998 & 3.999\\ \cline{2-5}
    & \textbf{C7} & 5.993 & 5.996 & 6.500 \\
    \cline{2-5}
    \end{tabular}
    \caption{k=2}
\end{subtable}

\vspace{1em}

\begin{subtable}[t]{0.45\textwidth}
    \centering
    \begin{tabular}{c|c|c|c|c|}
    \multicolumn{2}{r}{} & \multicolumn{3}{c}{\small\textbf{Refinement Index}} \\
    \cline{3-5}
    \multicolumn{2}{r|}{} & \textbf{4} & \textbf{5} & \textbf{6} \\
    \cline{2-5}
    \multirow{4}{*}{\rotatebox[origin=c]{90}{\small\textbf{Conv. Rate}}}
    & \textbf{F4} & 0.780 & 0.906 & 0.956 \\ \cline{2-5}
    & \textbf{C4} & 1.996 & 1.999 & 2.000 \\ \cline{2-5}
    & \textbf{C6} & 3.992 & 3.998 & 3.999 \\ \cline{2-5}
    & \textbf{C8} & 5.989 & 6.001 & 4.190 \\
    \cline{2-5}
    \end{tabular}
    \caption{k=3}
    \end{subtable}
\hspace{1em}
\begin{subtable}[t]{0.45\textwidth}
    \centering
    \begin{tabular}{c|c|c|c|c|}
    \multicolumn{2}{r}{} & \multicolumn{3}{c}{\small\textbf{Refinement Index}} \\
    \cline{3-5}
    \multicolumn{2}{r|}{} & \textbf{4} & \textbf{5} & \textbf{6} \\
    \cline{2-5}
    \multirow{4}{*}{\rotatebox[origin=c]{90}{\small\textbf{Conv. Rate}}}
    & \textbf{F5} & 1.067 & 1.044 & 1.023 \\ \cline{2-5}
    & \textbf{C5} & 1.997 & 1.999 & 2.000 \\ \cline{2-5}
    & \textbf{C6} & 1.992 & 1.998 & 2.000 \\ \cline{2-5}
    & \textbf{C7} & 3.994 & 3.995 & 3.358 \\ \cline{2-5}
    & \textbf{C8} & 3.988 & 3.997 & 4.005 \\ \cline{2-5}
    & \textbf{C9} & 5.978 & 4.004 & -0.430 \\
    \cline{2-5}
    \end{tabular}
    \caption{k=4}
\end{subtable}

\caption{Convergence rates for different FD approximations across various refinement indices for the $k^{th}$ derivative.}
\label{fig:convergence_tables}
\end{table}

\section{Conclusions}
\label{sec:conclusions}

Accuracy order of an FD approximation is obtained by analyzing the Taylor series expansion of the FD error
\[
	E(h) := \left| \FD^{k}f(\xstar) - f^{(k)}(\xstar)\right|.
\]
An FD approximation has accuracy order at least $r$ if $E(h)=\mathcal O(h^r)$ and exactly $r$ if $E(h)=\Theta(h^r)$.
It is well understood that an $N$-point FD approximation to the $k^{th}$ derivative can achieve an accuracy order of at least $N-k$.
We have established several theoretical results demonstrating when superconvergence (an accuracy order of $N-k+1$) occurs for centered FD approximations, which we summarize in Table~\ref{table:theory}.
Our work shows that the parity of the derivative order $k$ dictates the symmetry (or skew-symmetry) of coefficients $\{c_n\}$, and the parity of the number of stencil points $N$ in order for superconvergence to be achieved.
In particular, $k$ and $N$ must be of opposite parity and $\{c_n\}$ are symmetric for even $k$ and skew-symmetric for odd $k$.
The case with $k$ and $N$ odd was interesting in that such FD approximations have $c_0=0$.
So standard convergence of these centered approximations can be interpreted as superconvergence for an $(N-1)$-point stencil.

\begin{table}[h!]
\centering
\begin{tabular}{c|| c | c }
	& Odd $k$ & Even $k$ \\ \hline \hline
\multirow{2}{*}{Odd $N$} & skew-symmetric $\{c_n\}$ & symmetric $\{c_n\}$\\
	& $r=N-k$ & $r=N-k+1$ \\ \hline
\multirow{2}{*}{Even $N$} & skew-symmetric $\{c_n\}$ & symmetric $\{c_n\}$\\
	& $r = N-k+1$ & $r=N-k$ 
\end{tabular}
\caption{Summary of theoretical accuracy order $r$ of centered FD approximations for derivative order $k$, number of stencil points $N$, and coefficients $\{c_n\}$.}
\label{table:theory}
\end{table}

By narrowing the scope of our analysis to centered FD approximations, we were able to demonstrate the role of (skew-)symmetry of $\{c_n\}$ in obtaining superconvergence.
Our analysis should be contrasted with similar work by \cite{sadiq2014} where they take a more general approach by considering, what we call, balanced\footnote{See Remark~\ref{rm:balanced} where we defined \textit{balanced} stencils and motivate our use of \textit{centered} stencils.} and potentially complex stencils.
The importance (skew-)symmetry of FD coefficients is no longer evident in this general approach.
However, similar to \cite{sadiq2014}, our results are sharp in that we disprove the possibility of higher accuracy order beyond $N-k+1$ for superconvergent approximations and an order of $N-k$ for regular convergence in the appropriate cases.

\appendix

\section{Asymptotic Notation}\label{app:asymp}
Big-O notation is commonly used to quantify the decay of FD errors with respect to the stencil size $h>0$ as $h\to0^+$, and in particular introduce the concept of convergence rates and accuracy orders; see \cite{atkinson2004, leveque2007} for definitions of big-O notation in the context of finite differences.
We point out, however, that big-O notation only implies a lower bound on the convergence rate.
In order to state a sharper result, where we can prove the exact convergence rate, we make use of big-Theta notation; \cite{knuth1997}.
For clarity and completeness, we provide a definition of big-O and big-Theta notation here.
We also prove intermediary results that are used in our main theorems.

\begin{definition}\label{def:BigO}
Let $f(h)$ and $g(h)$ be two real valued functions defined over $[0,a)$, for some $a>0$.
We write
\[
	f(h) = \mathcal O(g(h)) \; \text{as} \; h\to 0^+
\]
if there is some constant $C>0$ such that
\[
	|f(h)| \le C|g(h)|
\]
for all $0<h<a$ sufficiently small.
\end{definition}

\begin{definition}\label{def:BigTheta}
Let $f(h)$ and $g(h)$ be two real valued functions defined over $[0,a)$, for some $a>0$.
We write
\[
	f(h) = \Theta(g(h)) \; \text{as} \; h\to 0^+
\]
if there are some constants $C_*,C^*>0$ such that
\[
	C_*|g(h)| \le |f(h)| \le C^*|g(h)|
\]
for all $0<h<a$ sufficiently small.
\end{definition}

\begin{lemma}\label{lem:BigTheta2}
Suppose the error of an FD approximation is of the form
\begin{equation}\label{eq:err_rs}
	E(h) = |K_1h^r + K_2(h)h^{r+s}|
\end{equation}
where $r,s>0$, and constant $K_1\neq 0$ is independent of $h$.
Assume there exists $0< K_2^*<\infty$ independent of $h$ such that
\[
	|K_2(h)| < K_2^*
\]
for $h>0$ small enough.
It follows that 
\[
	E(h) = \Theta(h^r).
\]
\end{lemma}
\begin{proof}
From (\ref{eq:err_rs}), the upper bound on $K_2(h)$, and the triangle inequality we have
\[
	\frac{E(h)}{h^r} 
	\leq |K_1| + h^s|K_2^*|.
\]
Let 
\[
	C^* := |K_1| + |K_2^*|.
\]
It follows that 
\[
	\frac{E(h)}{h^r}  \leq C^*
\]
for all $0<h\le 1$.

From the reverse triangle inequality, we have
\[
	|K_1 + h^s K_2(h)| \geq |K_1| - h^s |K_2(h)|.
\]
Then for
\[
	0<h\le  \left(\frac{|K_1|}{2K_2^*}\right)^{1/s}
\]
we have
\[
	|K_1 + h^s K_2(h)| \geq |K_1| - \frac{|K_1|}{2K_2^*} |K_2(h)| \ge \frac{1}{2}|K_1|.
\]
With
\[
	C_* := \frac{|K_1|}{2},
\]
we can conclude that 
\[
	\frac{E(h)}{h^r} \geq C_*.
\]

In summary, for 
\[
	0< h \le \min\left\{ 1, \left(\frac{|K_1|}{2K_2^*}\right)^{1/s} \right\},
\]
we have shown that
\[
	C_* \le \frac{E(h)}{h^r} \le C^* \; \implies \; C_*h^r \le E(h) \le C^*h^r.
\]
\end{proof}

\section{Vandermonde Matrices}\label{app:Vander}

In this section we introduce some useful definitions and results pertaining to Vandermonder matrices.
See \cite{golub2013} for more on Vandermonde matrices.

\begin{definition}\label{eq:Vander}
A {Vandermonde matrix}, $V\in\R^{M\times N}$, is a matrix of the form
\[
	[V]_{mn} = x_n^{m-1}, \quad m=1,2,...,M, \; n=1,2,...,N,
\]
for some $\{x_n\}_{n=1}^{N}\subset\R$.
\end{definition}

\begin{lemma}
For a square Vandermonde matrix $V\in\R^{N\times N}$, with entries $[V]_{mn} = x_n^{m-1}$, the determinant is given by
\[
	\det(V) = \prod_{1\leq i < j\leq N}(x_j-x_i).
\]
It follows that $V$ is invertible if and only if $x_i\neq x_j$ for all $i,j=1,2,...,N$.
\end{lemma}

The $A$ matrices that show up in our main theorems are all invertible, square Vandermonde matrices.
There are also $B$ matrices that we construct in our proofs that, although are not Vandermonde matrices in the traditional sense, are classified as generalized Vandermonde matrices with certain properties that guarantee invertibility, \cite{gantmacher1959,heineman1926}.

\begin{definition}\label{def:genVander}
A {generalized Vandermonde matrix}, $V\in\R^{M\times N}$, is a matrix of the form
\[
	[V]_{mn} = x_n^{z_m}, \quad m=1,2,...,M, \; n=1,2,...,N,
\]
for some $\{x_n\}_{n=1}^N\subset \R$ and $\{z_m\}_{m=1}^M\subset \N_0$ assuming
\[
	0\le z_1 < z_2 < \cdots < z_M.
\]
\end{definition}

\begin{lemma}\label{lem:genVander}
Let $V\in\R^{N\times N}$ be a generalized Vandermonde matrix with entries
\[
	[V]_{mn} = x_n^{z_m}, \quad m=1,2,...,M, \; n=1,2,...,N.
\] 
If $0<x_1 < x_2 < \cdots < x_N$, then $V$ is invertible.
\end{lemma}
\begin{proof}
See \cite{gantmacher1959}, example 1 in page 99, for a proof.
\end{proof}

\section*{Acknowledgments} 
 
 We would like to acknowledge Anthony Cortez, undergraduate alumni of Fresno State, for his initial work on this project.
In particular, Anthony had postulated the superconvergence theorem we have proved here along with the use of symmetric and skew-symmetric FD coefficients.
This project was funded by the US Department of Education through the STEAM: Enriched Pathways program, award number P031S200054.

 
 \newpage
{\footnotesize  
\medskip
\medskip
\vspace*{1mm} 
 
\noindent {\it Mario Javier Bencomo}\\  
California State University, Fresno,\\
Fresno, California, USA.\\
E-mail: {\tt bencomo@mail.fresnostate.edu}\\ \\  

\noindent {\it Joseph Igot}\\  
California State University, Fresno,\\
Fresno, California, USA.\\
E-mail: {\tt jigot@mail.fresnostate.edu}\\ \\  

\noindent {\it Emma Maltes}\\  
California State University, Fresno,\\
Fresno, California, USA.\\
E-mail: {\tt emmaltes@mail.fresnostate.edu}\\ \\  

}


\end{document}